\documentclass[12pt,a4paper]{amsart}
\usepackage{a4wide}
\usepackage[utf8]{inputenc}

\usepackage{amsmath,amssymb}

\usepackage{hyperref}
\hypersetup{
	colorlinks,
	citecolor=black,
	filecolor=black,
	linkcolor=black,
	urlcolor=black
}

\newtheorem{definition}{Definition}
\newtheorem{theorem}{Theorem}
\newtheorem{proposition}{Proposition}
\newtheorem{cor}{Corollary}
\newtheorem{claim}{Claim}

\newtheorem{fact}{Fact}
\newtheorem{remark}{Remark}
\newtheorem{example}{Example}

\def\ca{\mathcal{A}}
\def\cb{\mathcal{B}}
\def\cc{\mathcal{C}}
\def\cf{\mathcal{F}}
\def\ch{\mathcal{H}}
\def\ci{\mathcal{I}}
\def\cm{\mathcal{M}}
\def\cn{\mathcal{N}}

\def\c{\mathfrak{c}}

\def\bbq{\mathbb{Q}}
\def\bbr{\mathbb{R}}

\def\w{\omega}

\def\cantor{{2^\w}}
\def\2cantor{{\cantor\times \cantor}}
\def\cof{\rm cof}
\def\cov{\rm cov}

\def\then{\longrightarrow}
\def\iff{\longleftrightarrow}

\title{Nonmeasurable images}
\author{Aleksander Cieślak}
\address{Aleksander Cieślak, Faculty of Mathematics, Wrocław University of Science and Technology,
50-370 Wrocław, Poland} 
\email{aleksander.dxpody@gmail.com}

\author{Robert Rałowski}
\address{Robert Rałowski, Faculty of Mathematics, Wrocław University of Science and Technology,
50-370 Wrocław, Poland}
\email{robert.ralowski@pwr.edu.pl}


\date{\today}

\begin{document}

\begin{abstract} In this article we will investigate nonmeasurability with respect to some $\sigma$-ideals in Polish space $X,$  of images of subsets of $X$ by selected mappings defined on  the space $X$.
Among of them we answer the following question: "It is true that there exists a subset of the unit disc in the real plane such that the continuum many projections onto lines are Lebesgue measurable and continuum many projections are not?". It is known that there exists continuous function $f:[0,1]\to [0,1]$ such that for every Bernstein set $B\subseteq [0,1]$ we have $f[B]=[0,1].$ We show relative consistency with $ZFC$ of fact that the above result is not true for some $\cn$ or $\cm$-completely nonmeasurable sets, even if we take less than $\c$ many continuous functions. 
\end{abstract}

\maketitle

\section{Introduction}
We adopt set theoretical notations. For any fixed set $X$ by $P(X)$ we mean the powerset of $X$. By $\omega$ we denote least infinite ordinal number, $\kappa,\lambda$ stands for uncountable cardinal number and $\alpha,\beta,\gamma,\xi$ for infinite ordinals. By $[X]^{<\kappa}$ we denote the family of all subsets of $X$ which has cardinality less than $\kappa$.
Let $X$ be an uncountable Polish space then by $Bor(X)$ denote a $\sigma$-algebra generated by all open subsets of $X$ and by $Perf(X)$ we mean the collection of all perfect subsets of $X$. The family $\ci\subseteq P(X)$ forms an $\sigma$-ideal on $X$ with all singletons iff 
\begin{itemize}
    \item $(\forall A\subseteq B \in \ci)(A\in\ci)$,
    \item $(\forall (A_n)_{n\in\omega}\in\ci^\omega)(\bigcup_{n\in\omega} A_n \in\ci)$ and
    \item $(\forall x\in X)(\{ x\} \in\ci)$.
\end{itemize}
Moreover, an $\sigma$-ideal $\ci$ has Borel base if
$$
(\forall A\in\ci)(\exists B\in Bor(X)\cap\ci)(A\subseteq B).
$$

For any family $\cf\subseteq P(X)$ and $\sigma$-ideal $\ci$ let us recall cardinal coefficients as follows:
$$
\cov(\cf,\ci ) = \min \{ |\cf'|:\; \ca'\subseteq \land (\exists B\in Bor(X)\setminus \ci) (B\subseteq \bigcup\ca' ) \}
$$
If $\cf = \ci$ then we will use abbreviation $\cov_h(\ci)$ for $\cov(\cf,\ci)$. It is well known that if $\ci$ is one of the $\sigma$-ideals $\cm,\cn$ then $\cov_h(\ci)=\cov(\ci)$.
$$
Cof(\ci) = \min \{ |\cb|:\; \cb\subseteq Bor(X)\setminus \ci \land (\forall A\in Bor(X)\setminus\ci)(\exists B\in\cb)(B\subseteq A)\}
$$
and
$$
\cof(\ci)  = \min \{ |\cb|:\; \cb\subseteq\ci\land (\forall A\in\ci) (\exists B\in\cb)  (A\subseteq B)  \}.
$$
Cichoń-Kamburelis-Pawlikowski Theorem (see \cite{CKP}) says that if $\ci$ is $c.c.c.$ $\sigma$-ideal with Borel base then $\cof(\ci)=Cof(\ci).$

Let $X$ be a Polish space and $\ci$ be any $\sigma$-ideal with Borel base then subset $A\subseteq X$ is $\ci$-measurable if $A$ is member of $\sigma$-algebra generated by the set $Bor(X)\cup\ci$. Moreover, $A$ is completely $\ci$-nonmeasurable in $X$ if
$$
(\forall B\in Bor(X)\setminus\ci) (A\cap B \text{ is not } \ci - \text{measurable}).
$$
It is easy to see that $A$ is completely $\ci$-nonmeasurable if
$$
(\forall B\in Bor(X)\setminus\ci) (A\cap B \ne \emptyset\land A^c\cap B\ne\emptyset).
$$
We say that $A\subseteq X$ is a Bernstein set if
$$
(\forall B\in Perf(X)) (A\cap B \ne \emptyset\land A^c\cap B\ne\emptyset).
$$

For fixed family of sets $\ca\subseteq P(X)$ and any point $x\in X$ we define a star of $\ca$ at $x$ as follows:
$$
\ca(x) = \{ A\in\ca:\; x\in A \}.
$$

\section{Results}
\begin{theorem}\label{main-first} Let $X$ be a Polish space, $\{ Y_\alpha:\; \alpha<\c\}$ be a collection of Polish spaces and assume that $\c$ is regular cardinal number. If $\{ f_\alpha:\alpha<\c\}$ is a family of functions such that for any $\alpha<\c$
\begin{enumerate}
    \item $f_\alpha[X] = Y_\alpha$,
    \item for any $y\in Y_\alpha$ we have $f^{-1}_\alpha [y]\in [X]^{<\c}$.
\end{enumerate}
Then there exists a subset $A\subseteq X$ such that for any $\alpha<\c$ $f_\alpha[A]$ is a Bernstein set in $Y_\alpha$.
\end{theorem}
\begin{proof}
We enumerate $Perf(Y_{\alpha})=\{P_{\alpha}^{\gamma}:\alpha\leq\gamma<\c\}$ and by induction we build sequences $\{a_{\alpha}^{\gamma}:\alpha\leq\gamma<\c\}$ and $\{b_{\alpha}^{\gamma}:\alpha\leq\gamma<\c\}$ such that for each $\alpha,\gamma<\c$ we have:
\begin{enumerate}
    \item $a_{\alpha}^{\gamma}\in f_{\alpha}^{-1}[P_{\alpha}^{\gamma}]$
    \item $b_{\alpha}^{\gamma}\in P_{\alpha}^{\gamma}$
    \item $a_{\alpha}^{\gamma}\notin\bigcup\{f_{\xi}^{-1}[b_{\xi}^{\sigma}]:(\xi<\alpha) \lor (\xi=\alpha \land \sigma<\gamma)\}$
    \item $b_{\alpha}^{\gamma}\notin\bigcup\{f_{\xi}^{-1}[a_{\xi}^{\sigma}]:(\xi<\alpha) \lor (\xi=\alpha \land \sigma\leq\gamma)\}$  
\end{enumerate}
Construction is valid because the set $f_{\alpha}^{-1}[P_{\alpha}^{\gamma}]$ is of cardinality $\c$ which is regular. Thus by condition (3), it cannot be cover by the union of less then $\c$ sets of cardinality less then $\c$ and a point $a_{\alpha}^{\gamma}$ can be found. The same argument works for $b_{\alpha}^{\gamma}$.
Now the set $A=\{a_{\alpha}^{\gamma}:\alpha\leq\gamma<\c\}$ clearly works for thesis of the theorem.
\end{proof}

\begin{example}
Let $X=\omega^{\omega}$, $Y_{C}=\omega^{C}$ where $C\subset\omega$ belongs to filter of cofinite subsets of $\omega$ and let $f_{C}:\omega^{\omega}\longrightarrow \omega^{C}$ be a projection map i.e. $f_{C}(x)=x|_{C}$. Defined spaces and functions clearly fulfill the assumptions of previous theorem and thus exists $A\subset \omega^{\omega}$ such that each image $f_{c}[A]$ is a Bernstein subset of $\omega^{C}$.
\end{example}

\begin{remark} For $\c$ not being a regular cardinal, the proof presented above may fail even for two functions satisfying condition of the theorem.
\end{remark}
Assuming that $\c$ is not regular, let fix some nonempty perfect set $P\subset X$ and write it as a disjoint union $P=\bigcup_{\alpha<\lambda}P_{\alpha}$ where $\lambda< \c$ and for every $\alpha<\lambda$ we have $|P_{\alpha}|<\c$. Then we put $f,g:X\longrightarrow X$ where $g$ is identity on $X$ and $f$ is such that $f[P_{\alpha}]=y_{\alpha}\in P_{\alpha}$ for each $\alpha<\lambda$ and $f|_{X\setminus P}:X\setminus P\longrightarrow X\setminus P$ is a bijection. Now with bad luck in choosing point and in enumeration if $\{y_{\alpha}:\alpha<\lambda\}\subset \{b_{\alpha}^{\gamma}:\alpha\leq\gamma<\lambda\}$ and perfect set $P$ occours in enumeration with index $>\lambda$ then it is impossible to find point $a\in P$ in previous proof.

\begin{remark} If we consider any function $f:X\to X_0$ such that $f[X]$ is a Polish space and any Bernstein set $A\subseteq X$ then 
	\begin{enumerate}
		\item if preimage of any singleton of $f[X]$ contains a perfect set then $f[A] = f[X]$,
		\item if $f$ is continuous then $f[A]$ contains some Bernstein set in $f[X]$ (because any preimage of perfect set in $f[X]$ contains perfect set in $X$).
	\end{enumerate}
\end{remark}

Now we consider a families of functions which necessary not have small preimages of singletons as in the Theorem \ref{main-first}. 
We introduce the cardinal coefficient $Cov(\cf,\mathfrak{I}).$ Namely, let $X\subseteq X_0$ are fixed Polish spaces. Let $\cf\subseteq X_0^X$ and $\mathfrak{I} = \{ \ci_f: f\in\cf\}$. We say that the pair $(\cf,\mathfrak{I})$ is fine in $X\subseteq X_0$ if for any $f\in\cf$ we have
\begin{itemize}
	\item $f[X]$ is Polish subspace of $X_0$,
	\item $\ci_f$ is $\sigma$-ideal with Borel base on $f[X]$ containing all singletons of $f[X]$.
\end{itemize}
Then we define 
$
Cov(\cf,\mathfrak{I}) = \min\{ |Z|: Z\subseteq X_0\land (\exists f\in\cf)(\exists B\in Bor(f[X])\setminus \ci_f)(\exists \cf_0\subseteq\cf) (|\cf_0|\le |Z| \land f^{-1}[B]\subseteq \bigcup \{ h^{-1}[Z]: h\in\cf_0\})\}
.$

\begin{theorem}\label{main-two} Let $\ci$ be $\sigma$-ideal with Borel base on the fixed Polish space $X_0$ and $X\subseteq X_0$ be a Polish space also. Let $(\cf,\mathfrak{I})$ be a fine family in $X\subseteq X_0$ with the following properties:
\begin{enumerate}
    \item $|\cf| \le \sup\{ Cof(\ci_f): f\in\cf\}$,
    \item $\sup \{ Cof(\ci_f):f\in\cf\} \le Cov(\cf,\mathfrak{I}).$
\end{enumerate}
Then there exists subset $A$ of $X$ such that for any $f\in\cf$ the image $f[A]$ is completely $\ci_f$-nonmeasurable in $f[X]$.
\end{theorem}
\begin{proof}  First of all let us define $\kappa = \sup\{ Cof(\ci_f): f\in\cf\}$.
Then enumerate our family $\cf=\{ f_\alpha:\alpha< \kappa \}$ in such way that every member of $\cf$ appears $\kappa$ many times (because $|\cf|\le \kappa$) and $\cb_\alpha = \{ B\in Bor(X)\setminus \ci_{f_\alpha}:\; B\subseteq f_\alpha[X]\}$ as $\{ B_\xi^\alpha:\; \xi<\kappa\}$ be a Borel base of the $\sigma$-ideal $\ci_f$ on $f_\alpha[X]$ for every $\alpha < \kappa$. Moreover, let us assume that if $f_\alpha = f_\beta$ then the enumerations are identical i.e. $B_\xi^\alpha = B_\xi^\beta$ for all $\xi<\kappa$.

By transfinite recursion let us define sequence $\{  (a_\xi^\alpha,d_\xi^\alpha): \xi < \alpha < \kappa  \}$ satisfying the following conditions: for any $\alpha < \kappa$:
\begin{enumerate}
    \item for any $\xi < \alpha$ we have $f_\alpha(a_\xi^\alpha),d_\xi^\alpha\in B_\xi^\alpha$,
    \item $\{f_\delta(a_\xi^\beta):\xi < \beta \le \delta\le \alpha\}\cap \{ d_\xi^\beta:\; \xi < \beta \le \alpha\} = \emptyset$,
\end{enumerate}	
First we show that the set $A= \{ a_\xi^\alpha: \xi<\alpha <\kappa \}$ is required one. For $D=\{d_\xi^\alpha: \xi<\alpha<\kappa \}$ the following conditions are true:
\begin{itemize}
	\item for any $f\in\cf$ and $B\in Bor(f[X]) \setminus \ci_f$ $f[A]\cap B\ne\emptyset$ and $B\cap D\ne\emptyset$,
	\item for any $f\in\cf$ $f[A]\cap D = \emptyset$.
\end{itemize}
To see that the first bullet is true, choose $f\in\cf$ and $B\in Bor(f[X]) \setminus \ci_f$.  Next, there are $\xi,\alpha<\kappa$ such that $f=f_\alpha$ and $B=B_\xi^\alpha$. If $\alpha\le \xi$ then there is $\alpha_0<\kappa$ such that $\xi < \alpha_0$ and $f=f_\alpha=f_{\alpha_0}$ ($f$ appears $\kappa$ many times in the enumeration of $\cf$). Then $d_\xi^{\alpha_0}\in B_\xi^{\alpha_0}=B_\xi^\alpha = B$ and $f(a_\xi^{\alpha_0}) = f_{\alpha_0} (a_\xi^{\alpha_0}) \in B_\xi^{\alpha_0} = B_\xi^\alpha = B$.

Assume that the last condition is not true. Then $f[A]\cap D\ne\emptyset$ for some $f\in\cf$. Then there are $\alpha,\alpha',\xi',\alpha",\xi"<\kappa$ such that $f=f_\alpha,$ $\xi'<\alpha',$ $\xi" < \alpha"$ and $f_\alpha(a_{\xi'}^{\alpha'}) = d_{\xi "}^{\alpha " }$. But $f$ appears $\kappa$ many times in enumeration of the family $\cf$, thus there is $\alpha_0<\kappa$ such that $\alpha,\alpha',\alpha" < \alpha_0$ and $f=f_\alpha = f_{\alpha_0}$. 
$$
d_{\xi"} ^{\alpha" }= f(a_{\xi'}^{\alpha'}) = f_\alpha(a_{\xi'}^{\alpha'}) = f_{\alpha_0} (a_{\xi'}^{\alpha'}) 
$$
what is impossible by $(2)$ and the fact that $\xi' < \alpha' < \alpha_0$ and $\xi" < \alpha" < \alpha_0$.

We will show the validity of our transfinite construction. Assume that we have sequence of the length $\gamma$ (where $\gamma < \kappa$) with the above conditions i.e. for every $\alpha <\gamma$ the condition $1)$ and $2)$ are satisfied. It is enough to show that there exists sequence $\{(a_\xi^\gamma,d_\xi^\gamma):\xi<\gamma\}$ such that the following conditions 
\begin{enumerate}
    \item for any $\xi<\gamma$ we have $f_\gamma(a_\xi^\gamma),d_\xi^\gamma\in B_\xi^\gamma$ and
    \item $\{f_\delta(a_\xi^\beta):\xi<\beta \le \delta \le \gamma\}\cap \{ d_\xi^\beta:\; \xi < \beta \le \gamma\} = \emptyset$
\end{enumerate}	
are true. To do this let us assume that for fixed $\eta<\gamma$ we have sequence $\{(a_\xi^\gamma,d_\xi^\gamma):\xi<\eta\}$. Let $Z=\{ d_\xi^\alpha: \xi\le\alpha <\gamma\}\cup \{ d_\xi^\gamma: \xi<\eta\}$ and let $\cf_0 = \{ f_\alpha:\; \alpha<\gamma\land f_\alpha\ne f_\gamma\}$. Then $|\cf_0|$ and $|Z|$ have size less than $\kappa$. By the last assumption in our Theorem we can find point $a\in f_\gamma^{-1}[B_\eta^\gamma]\setminus \bigcup\{ f^{-1}[Z]: f\in \cf_0\}$. Next, we can find point $d\in B_\eta^\gamma$ such that $d\notin \{ f_\gamma(a)\}\cup \{ f_\delta(a_\xi^\beta): (\xi<\eta\land \beta = \delta = \gamma) \lor (\xi < \beta\le \delta <\gamma)\}$. Set $a_\eta^\gamma = a$ and $d_\eta^\gamma = d$. Then validity of our construction has been proved.
\end{proof}

In many cases we want to known some properties of the set $A$ in the {\em Theorem \ref{main-first}} and {\em Theorem \ref{main-two}}.

\begin{theorem}\label{main-ideals} Assume that $\ci$ is $\sigma$-ideal with Borel base on the fixed Polish space $X_0$ and let $X\subseteq X_0$ be $\ci$-positive Borel subset. Let $(\cf,\mathfrak{I})$ be a fine family in $X\subseteq X_0$ such that
\begin{enumerate}
  \item $|\cf|\le  \max\{ Cof(\ci), \sup\{ Cof(\ci_f) : f\in \cf \}  \}$,
  \item there is set $Z\in \ci$ such that $Cof(\ci) \le cov(\{f^{-1} [\{d \}]:\; f\in\cf \land d\in X_0\setminus Z \}, \ci )$,
  \item $\max\{ Cof(\ci),\sup \{ Cof(\ci_f):f\in\cf\}\} \le Cov(\cf,\mathfrak{I}).$
  
\end{enumerate}
Then there exists $A\subseteq X$ which is completely $\ci$-nonmeasurable in $X$ such that for every $f\in\cf$ an image $f[A]$ is completely $\ci$-nonmeasurable in $f[X]$.
\end{theorem}
\begin{proof} First of all let $\kappa = \max\{ Cof(\ci),\sup \{ Cof(\ci_f):f\in\cf\}\}$. Let us enumerate our family $\cf=\{ f_\alpha:\alpha< \kappa \}$ in such way that every member of $\cf$ appears $\kappa$ many times (because $|\cf|\le \kappa$) and let $\cb_\alpha = \{ B\in Bor(X)\setminus \ci_{f_\alpha}:\; B\subseteq f_\alpha[X]\}$ as $\{ B_\xi^\alpha:\; \xi<\kappa\}$ be a Borel base of $\ci_{f_\alpha}$ in $f_\alpha[X]$  for every $\alpha < \kappa$. Moreover, let assume that if $f_\alpha = f_\beta$ then the enumerations of $\cb_\alpha$ and $\cb_\beta$ are identical. Next, let us enumerate all $\ci$-positive Borel subsets of $X$ as follows $\{C_\xi:\xi<\kappa\}$.
	
Now by transfinite recursion let us define sequences $\{  (a_\xi^\alpha,d_\xi^\alpha): \xi < \alpha < \kappa  \}$ and $\{ (b_\alpha,c_\alpha): \alpha<\kappa  \}$ with the following conditions: for every $\alpha<\kappa$ we have
	\begin{enumerate}
		\item $b_\alpha,c_\alpha\in C_\alpha$
		\item $(\{a_\xi:\xi<\alpha \} \cup \{b_\xi:\xi<\alpha  \}) \cap \{c_\xi:\xi<\alpha\}  = \emptyset $, 
		\item for any $\xi<\alpha$ we have $f_\alpha(a_\xi^\alpha),d_\xi^\alpha\in B_\xi^\alpha$,
		\item $(\{f_\delta(a_\xi^\beta):\xi < \beta \le \delta \le \alpha\} \cup \{f_\delta(b_\xi^\beta):\xi < \beta\le \delta\le \alpha\} )\cap \{ d_\xi^\beta:\; \xi < \beta \le \alpha\} = \emptyset.$
	\end{enumerate}	
Arguing as in the proof of {\em Theorem \ref{main-two}}, the second condition holds by using $3$-th assumption and condition $4$ above which is valid by using $4$-th assumption of the our theorem. The set $A = \{ a_\alpha:\alpha<\c \} \cup \{ b_\alpha:\alpha<\kappa  \}$ is required one.
\end{proof}

\begin{cor} Assume MA. If $\ci \in \{ \cn,\cm\}$ is a $\sigma$-ideal defined on Cantor space and $X\subset 2^\w$ be a Borel $\ci$-positive. If $\cf$ is with size at most equal $\c,$ if for any $f\in\cf$ $rng(f)$ is Borel and $\ci_f\in\{\cn,\cm\}$ then Theorem \ref{main-two} and Theorem \ref{main-ideals} are true. 
\end{cor}

In above Corollary the uniformity of $\cn$ and $\cm$ are equal $\c$. Then the sets in assertion has size also $\c$. 

In the next proposition we show that consistently with ZFC this set can has size less than $\c$.
\begin{proposition} It is relatively consistent with ZFC that for any $\cn$-positve Borel subset $C$ of $2^\w$ and any family $\cf$ such that
\begin{itemize}
    \item each $f\in\cf$ is Borel function, $C\subseteq dom(f)$ and $rng(f)\subseteq 2^\w$,
    \item for each $f\in\cf$ the image $f[C]$ is Borel set with positive Haar measure,
    \item for each $f\in\cf$ and each $B\subseteq rng(f)$ such that $B\in Bor\setminus\cn$ the preimage $f^{-1} [B] \notin \cn$
\end{itemize}
there is $A$ of size less than $\c$ such that  any $f\in\cf$ $f[A]$ is completely $\cn$-nonmeasurable in $C$ and $f[C]$ respectively. Moreover, each set of the cardinality less than $\c$ is meager.
\end{proposition}
\begin{proof} Let $V$ be a ground model which satisfies $CH$. Let us consider any generic extension $V[G]$ obtained by adding $\w_2$ independent random reals $\{r_\xi:\xi<\w_2\}$. Now let us consider any non null Borel set $C$ and family $\cf$ which fulfills the above assumption.  Set $A=C\cap \{ r_\xi+q_n:\; \xi<\w_1\land n\in\omega\}$ where $Q=\{q_n:n\in\w\}$ is set of all rational numbers in $2^\w$ (all sequences of all but finitely many zeros). Let us observe that $A$ is completely $\cn$-nonmeasurable in $B$. Indeed, let $B\in Bor\setminus\cn$ be a subset of $C$ then by the fact that $\{r_\xi:\xi<\w_1\}\notin\cn$ the set $B\setminus \{r_\xi:\xi<\w_1\}$ contains some clopen set. Then there is $q\in Q$ such that $q\in B\setminus\{r_\xi:\xi<\w_1\}$ and then $A\cap B\ne\emptyset$. On the other hand, $A$ has cardinality $\aleph_1<\c$ so $B\setminus A\ne\emptyset.$ We have shown that $A$ is completely $\cn$-nonmeasurable in $C$.
	
Now we show that for any function $f\in\cf$ the image $f[A]$ is completely $\cn$-nonmeasurable in $f[C].$ Let $B$ be an arbitrary $\cn$ positive Borel subset of $f[C]$. Then $f^{-1}[B] \notin \cn$ is Borel non-null subset of $C$. But $A$ is completely $\cn$-nonmeasurable in $C$ thus there is $a\in A\cap f^{-1}[B]$. Then $f(a) \in B$ for some $a\in A$. We have $f[A]\cap B\ne \emptyset$ but $|f[A]| \le \aleph_1<\aleph_2=\c$. As $B$ was choosen arbitrary $f[A]$ is completely $\cn$-nonmeasurable in $f[C]$.
	
It is well known that in the random extension $V[G]$ $non(\cm) = \c$ so $[2^\w] ^{<\c} \subseteq \cm$.
\end{proof}

\begin{proposition} It is relatively consistent with ZFC that for any $\cm$-positive Borel subset $B$ of $2^\w$ and any family $\cf$ of size at most equal to $\c$ there is $A$ of size less than $\c$ such that for any $f\in\cf$ $f[A]$ is completely $\cm$-nonmeasurable in $B$ and $f[B].$ 
\end{proposition}

\begin{example}
Let $\mathcal{A}$ be a almost disjoint family on $\omega$ of cardinality $\c$. For $A\in \mathcal{A}$ let $\pi_{A}:\omega^{\omega}\longrightarrow \omega^{A}$ be such that $\pi_{A}(x)=x|_{A}$ and let $\mathcal{I}_{A}$ as well as $\mathcal{I}$ be $\sigma-$ideals of meager sets.
\end{example}

The following Corollary of Theorem \ref{main-first} gives negative answer for question given in \cite{mathoverflow-1}:
\begin{quote}[asked Aug 3 '11 at 7:51 simon 162]
{\em Suppose $A$ is contained in the unit square of $R^2$, and the projection of A on any line outside the unit square is not Lebesgue measurable in $R$. Does that imply that $A$ is not Lebesgue measurable in the plane?}
\end{quote}

Moreover, our answer is valid for measure and category simultaneously.
\begin{cor} There is subset $A\subset S^1$ of the unit circle that for any projection $\pi$ on real line $l \subseteq \bbr^2$ on the real plane of the set $A$ is a Bernstein set in $\pi[S^1]$.
\end{cor}

But from the Theorem \ref{main-ideals} we can prove the following Corollary.
\begin{cor} There is subset $A\subset D$ of the unit disc which is completely $\cn$-nonmeasurable in $D$ and such that for any projection $\pi$ on real line $l \subseteq \bbr^2$ on the real plane of the set $A$ is a completely $\cn$-nonmeasurable in $\pi[S^1]$. The same is true for category case.
\end{cor}

In the same webpage \cite{mathoverflow-1} user Gowers gives positive answer for the following question
\begin{quote}[Gerald Edgar Aug 3 '11 at 13:57]
{\em (a) All projections but two are non-measurable? Or: (b) Projections in uncountably many directions measurable and projections in uncountably many other directions non-measurable?}
\end{quote}

The user of Mathoverflow asked:
\begin{quote}[answered Aug 3 '11 at 14:47 gowers]
 {\em I don't know what happens if we ask for continuum many measurable projections and continuum many non-measurable projections ...}
\end{quote}

We have simple answer for the last question:
\begin{fact}\label{fakt-1} There exists subset $A$ of the unit disc $D=\{ (x,y)\in\bbr^2:\; x^2 + y^2 \le 1\}$ such that $\c$ many projections on the real $l$ line are not Lebesgue measurable and do not have Baire property in $l$ and $\c$ many projections on the real line $l$ are Lebesgue measurable and have Baire property in the line $l$. 
\end{fact}
\begin{proof} It is enough to define $A\subseteq D$ as follows:
$$
A = (B \times \{0\} \cap \bigg[-\frac{\sqrt{3}}{4},\frac{\sqrt{3}}{4}\bigg] \times \{0\} )\cup (\bbr \times \bigg\{ \frac{1}{2} \bigg\} \cap D),
$$
where $B\subseteq \bbr$ is a Bernstein set in the real line $\bbr$. Here we have $A\in \cn(\bbr^2)\cap \cm(\bbr^2)$.
\end{proof}

\begin{remark} The family of all projections are indexed by elements of $[0,\pi]$ of directions. And then in the above Fact \ref{fakt-1} the set of directions $\alpha\in [0,\pi]$ for which $\pi_\alpha[A]$ is measurable,  has positive Lebesgue measure in $[0,\pi],$ but less than $\pi$.
\end{remark}

In contrast to the previous Remark, we show the following Theorem.
\begin{theorem}\label{circle-parametrisation}  Let $X$, $Y$ and $Y_{\alpha}$ for $\alpha\in Y$ be a polish spaces, $\c$ be regular and $\{f_{\alpha}:\alpha\in Y\}$ a family of functions such that for all distinct $\alpha,\beta \in Y$ we have 
\begin{itemize}
    \item $\forall y\in Y_{\alpha}$ $|f_{\alpha}^{-1}[y]|=\c$
    \item $\forall y\in Y_{\alpha}$ and $y'\in Y_{\beta}$ $|f_{\alpha}[y]\cap f_{\beta}[y']|<\c$
\end{itemize}
Then there exists a subset $A\subseteq  X$ and disjoint Bernstein sets $F,G\subseteq Y$ such that $Y=F\cup G$ and
$$
F=\{ \alpha\in Y: f_\alpha[A]=Y_{A}\}$$
$$
G=\{ \alpha \in Y: f_\alpha[A] \text{ is Bernstein in }Y_{\alpha}\}.
$$
\end{theorem}
\begin{proof} Let us enumerate the following sets:
	\begin{itemize}
	    \item $Y=\{x_{\alpha}:\alpha<\c \}$
		\item $Y_{\alpha} = \{ y_{\alpha}^{\gamma}: \alpha\leq \gamma <\c  \}$
		\item $ Perf(Y_{\alpha})= \{ P_{\alpha}^{\gamma}: \alpha\leq \gamma <\c \}$
		\item $ Perf(Y)=\{P_{\alpha}:\alpha<\c\}$
	\end{itemize}
	Now, by transfinite recursion we define
	$$
	\{  (A_\xi, B_\xi, F_\xi, G_\xi, ) : \xi < \c  \}
	$$
	such that for any $\xi,\gamma<\c$ we have
	\begin{enumerate}
		\item $A_\xi\subset A_\gamma, B_\xi\subset B_\gamma, F_\xi\subset F_\gamma, G_\xi\subset G_\gamma $ for $\xi<\gamma$
		\item $|A_\xi|=|B_\xi|=|F_\xi|=|G_\xi|=\omega|\xi|$
		\item $A_\xi \subset X, B_\xi \subset \bigcup_{\alpha\in Y}Y_{\alpha}$
		\item $F_\xi, G_\xi \subset Y$
		\item $F_\xi\cap G_\xi =\emptyset$
		\item $x_\xi \in F_\xi \cup G_\xi$
		\item $F_\xi \cap P_\xi\neq \emptyset \neq G_\xi \cap P_\xi$
		\item $(x_\gamma \in F_\xi \land \gamma\leq\xi)\longrightarrow(\{y_{\gamma}^{\sigma}:\gamma\leq \sigma \leq \xi\}\subset\pi_{x_{\gamma}}[A_\xi])$
		\item $(x_\gamma \in G_\xi \land \gamma\leq\xi) \longrightarrow (f_{x_{\gamma}}[A_\xi]\cap B_\xi=\emptyset)$
		\item $(x_\gamma \in G_\xi \land \gamma\leq\xi)\longrightarrow(f_{x_{\gamma}}[A_\xi]\cap P_\gamma^\sigma\neq\emptyset\text{ for each }\gamma\leq \sigma \leq \xi)$
		\item $(x_\gamma \in G_\xi \land \gamma\leq\xi)\longrightarrow(\forall \gamma\leq \sigma \leq \xi P_\gamma^\sigma\cap B_\xi\neq \emptyset)$
		\item $(x_\gamma \in F_\xi \land \gamma\leq\xi)\longrightarrow(\forall \gamma\leq \sigma \leq \xi P_\gamma^\sigma\cap B_\xi= \emptyset)$
	\end{enumerate}
	The construction goes as follows: In step $\xi<\c$ firstly - if possible- we decide whether $x_\xi\in F_\xi$ or $x_\xi\in G_\xi$. Then we choose point for $F_{\xi}$ and point for $G_{\xi}$ such that $F_\xi\cap P_\xi \neq \emptyset$ and $ G_\xi\cap P_\xi \neq \emptyset $ paying attention to condition (9). Next for each $\gamma<\xi$ with $x_{\gamma}\in F_{\xi}$ we find a point $a\in f^{-1}_{\gamma}[y_{\gamma}^{\xi}]$ such that (9) holds. Next for each perfect set $P_{\gamma}^{\xi}$ such that $\gamma\leq\xi$ and $x_{\gamma}\in G_{\xi}$ we find point $a^{*}\in P_{\gamma}^{\xi}$ and then point $a\in f_{\gamma}^{-1}[a^{*}]$ satisfying (9). All such choosen points we put to $A_{\xi}$ and all these points can be found because $\pi_{\gamma}^{-1}[a^{*}]$ is of cardinality $\c$ which is regular, $|B_{\xi}|<\c$ and intersection of any two is of cardinality less then $\c$. Next for each perfect set $P_{\gamma}^{\xi}$ such that $\gamma\leq\xi$ and $x_{\gamma}\in G_{\xi}$ we find point $b\in P_{\gamma}^{\xi}\
setminus f_{\gamma}[A_{\xi}]$ and all these we put to $B_{\xi}$.
	
	Thus all points necessary for construction can be chosen.  Now let
	\begin{itemize}
		\item $A = \bigcup_{\xi<\c} A_\xi$,
		\item $F = \bigcup_{\xi<\c} F_\xi$ 
		\item $G = \bigcup_{\xi<\c} G_\xi$
	\end{itemize}
	and we have the following:  (5),(6) implies that $F$ and $G$ are decomposition of $Y$;  (5),(7) implies that $F$ and $G$ are Bernstein sets;  (8),(12) implies that for any $x\in F$ we have $f_{x}[A]=Y_{\alpha}$; finally conditions (9),(10),(11) implies that for any $x\in G$ we have $f_{x}[A]$ is a Bernstein set.
\end{proof}
\begin{remark}
The above proof works if we loose the assumption of regularity of $\c$ but assume that $f_{\beta}^{-1}[y]\cap f_{\alpha}^{-1}[y']$ is countable for any $\alpha,\beta<\c$ and distinct $y\in Y_{\alpha}$ and $y'\in Y_{\beta}$.
\end{remark}

\begin{remark} In the Theorem \ref{circle-parametrisation} we can replace real plane by $n$-dimensional euclidean space $n\ge 2$, unit disc $D\subseteq \bbr^2$ by unit ball, tangent lines to sphere $S^1$ in $\bbr^2$ by $n-1$-dimensional hyperplanes and finally segment $[0,\pi]$ by $n-1$-dimensional rectangle which parameterised unit sphere $S^{n-1}\subseteq \bbr^n$ of the our unit ball. 
\end{remark}

\begin{fact}\label{lusin} Let $n\ge 2$ be a fixed integer. Then every projection $\pi$ of the Lusin set $A\subseteq B(0,1) \subset \bbr^n$ into tangent hyperplane $l$ to $B(0,1)$ is Lusin set in $\pi[B(0,1)]$. The same result is true if we replace Lusin set by Sierpiński set.
\end{fact}
\begin{proof} We show this fact for Lusin set only. Let $A\subseteq B(0,1)$ be Lusin set and $P\subseteq \pi[B(0,1)]$ be a meager set. Then using Kuratowski-Ulam Theorem $B(0,1) \cap \pi[P]$ is also meager subset of the unit ball $B(0,1)$. Then the intersection $A\cap \pi^{-1}[P]$ is countable and thus $\pi[A] \cap P$ is countable also.
\end{proof}

Using the above Fact \ref{lusin} we can show the following consistency result.
\begin{fact}\label{coheny} It is relatively consistent with $ZFC$ theory that $\neg CH$ and for every integer $n\ge 2$ there exists Baire nonmeasurable subset $A$ of the cardinality less than $\c$ of the unit ball $B\subseteq \bbr^n$ such that projection $\pi[A]$ into any tangent to $B$ hyperplane does not have Baire property. The same result is true in the case of Lebesgue measure. 
\end{fact}
\begin{proof} Let $V\models GCH$ be any transitive model of $ZFC$ theory. Then adding $\omega_2$ independent Cohen reals $\cc_{\omega_2}$ to $V$ then $A=B\cap (\{ c_\xi:\xi<\omega_1\} + \bbq^n )$ is required set (here $\bbq$ are all rational numbers).
\end{proof}


Alexander Osipov asked \cite{AO} question "It is true that for every Bernstein set in $\bbr$ there are countable many continuous functions $\{ f_n:n\in\w\}$ such that $\bigcup_{n\in\w} f_n[B] = \bbr.$"

The answer is given by \cite{CMR}.
\begin{theorem}\label{osipov} There is function $f:[0,1]\to [0,1]$ such that for every Bernstein set $B\subseteq [0,1]$ the following is true $f[B] = [0,1].$
\end{theorem}

The function $f$ used in proof of the Theorem \ref{osipov} is far of differentiable function. In the same article \cite{CMR} has been proven  that the answer for Osipov question is false in case of function which fulfills so called condition $T_2$. Here we say that real function is $T_2$ if
$$
\{ y\in\bbr:\; |f^{-1}[\{y\}]| = \c\} \in \cn.
$$
\begin{theorem}\label{funkcje-T2} There exists a Bernstein set $A\subseteq \bbr$ such that $\bigcup\{f[A]:f\in\cf\}\ne\bbr$ for every countable collection $\cf$ of continuous functions satisfying the condition $T_2.$
\end{theorem}

Now we want to prove some independent result of $ZFC$ theory which is analogous to the Theorem \ref{funkcje-T2}.

We recall the definition of completely nomeasurablility respect to the $\sigma$-ideal.

\begin{definition} Let $X$ be Polish space and $I\subseteq P(X)$ is $\sigma$-ideal with Borel base with singletons. We say that set $A\subseteq X$ is completely $I$-nonmeasurable if
$$
(\forall B\in Bor(X)\setminus I)\; A\cap B\ne \emptyset \land A^c\cap B\ne \emptyset.
$$
\end{definition}
Every completely $[X]^{\le\w}$-nonmeasurable set is Bernstein set and vice-versa. If $A\subseteq [0,1]$ is completely $\cn$-nonmeasurable then inner measure is zero and outer measure is equal to one. If $A$ is completely $\cm$-nomeasurable then in every nonempty set $U$ the set $A\cap U$ has no Baire property in $U.$

We say that $A\subseteq [0,1]$ is strongly null if for every sequence of positive reals $(\epsilon_n)_{n\in\w}$ there is sequence of intervals $(I_n)_{n\in\w}$ in $[0,1]$ such that
$$
A\subseteq \bigcup_{n\in\w} I_n \;\; \land \;\; (\forall n\in\w) (\lambda(I_n) <\epsilon_n).
$$
It is well known that in Laver model every strongly null set is countable. In [GMS] authors have proven that $X$ is strongly null if and only if for each meager set $M$ we have
$X+M\ne[0,1]$. This led to following definition.
We say that set $X\subseteq[0,1]$ is strongly meager if for every set $N$ of measure zero we have $N+X\ne[0,1]$. It is well known that in Cohen model strongly meager sets are countable.

Now we are ready to formulate the main theorem of this subsection.

\begin{theorem}\label{main-osipov} It is consistent with $ZFC$ that there exists $A\subseteq \bbr$ and $B\subseteq\bbr$ such that
\begin{itemize}
    \item $A$ is strongly null,
    \item $B$ is completely $\cm$-nonmeasurable.
\end{itemize}
and for every family of continuous functions $\cf$, $|\cf|<\c$ satisfying $\forall f\in\cf$ $B\subseteq dom(f)$ we have that $A\setminus \bigcup\{ f[B]:\; f\in\cf\}\ne\emptyset.$ 
\end{theorem}


\begin{proof} We work on $\cantor$. Let $V$ be ground model satisfying CH. 
Consider the forcing notion $(\cc,\le)$ which adds $\w_2$ many independent Cohen reals. Let us split $\w_2$ onto two subsets $A_0$ and $B_0$ with the same cardinality equal to $\w_2$. Next define 
$$
A = \{c_\xi:\xi\in A_0\};\; \text{ and }\;\; B = \{ c_\xi:\; \xi\in B_0 \} + Q,
$$
where $Q$ is collection of all rationals in $\cantor$ (all sequences of all but finitely many zeros).
Then we have the following Claim.
\begin{claim} Let $G\subseteq \cc$ be generic filter over $V$.Then in $V[G]$ we have
\begin{enumerate}
    \item $A,B$ are Łusin sets and $B$ is dense,
    \item $B$ is completely $\cm$-nonmeasurable in $\cantor,$
    \item $A$ is strongly null set.
\end{enumerate}
\end{claim}
\begin{proof}[Proof of the Claim] 
From standard argument with nice names and factorisation lemma follows that for every meager $F_\sigma$ set $F$ the intersection $F\cap \{c_\xi:\; \xi<\w_2\}$ is countable. Then $A$ and $B$ are Lusin sets and then the first condition is proved. The second condition follows directly from the first (because $B$ is Lusin and dense in $\cantor$). To see that the last condition is true fix $(\epsilon_n)_{n\in\w}$. Let $(\epsilon_n)_{n\in\w}$ be any sequence of positive reals. Define $\{ s_n:\; n\in \w\}\subseteq 2^{<\w}$ and
$$
U = \bigcup\{ [s_n]:\; n\in\w\}.
$$
such that for ever $n\in\w$ we have
\begin{enumerate}
    \item $\frac{1}{\epsilon_{2n}} < |s_n|,$
    \item for every $s\in \w^{<\w}$ there is $m\in\w$ such that $s\subseteq s_m.$
\end{enumerate}
Then $U$ is dense open set. Then as $A$ is Lusin set we have that $A\setminus U$ is countable. It follows that except for countably many points $A$ is coveres by a meager set. This countably many points can be covered by the sequence of clopens $([t_n])_{n\in\w}$ with property $\lambda(t_n) < \epsilon_{2n+1}$ for every $n\in\w.$ Then we have
$$
A = (A\cap U)\cup (A\setminus U)\subseteq \bigcup_{n\in\w} [s_n] \cup \bigcup_{n\in\w} [t_n].
$$
Thus the proof of the Claim is finished.
\end{proof}
Now in $V[G]$, let $\cf$ be a family of size less than $\w_2$ of continuous functions $f\in\cf$ defined on dense $G_\delta$ sets $G_f = dom(f)$ such that $B\subseteq G_f.$ Each $f$ has the Borel code $c_f$ which is represented by some nice name $\dot{c_f}.$ Let observe that each name $c_f$ uses only countable many conditions $p\in\cc$. Denote them by $X_f$. Then for each $f\in\cf$ $c_f \in V^{\cc_{X_f}}$ where $X_f\in [\w_2]^{\le \w}.$ But $\w_2$ is regular cardinal and $|\cf|<\w_2$ then $\w_2 \setminus \bigcup_{f\in\cf} X_f$ has size equal to $\w_2.$ Then there exists $B_\cf\subseteq \w_2$ such that $B_0\subseteq B_\cf$ and $|A\setminus B_\cf| = \w_2$ and $\cf \in V[G_{B_\cf}]$ where 
$$
G_{B_\cf} = \{p\in\cc_{B_\cf}:\; (\exists q\in \cc_{\w_2\setminus B_\cf}\; (p,q)\in G).
$$ 
But $V[G_{B_\cf}]$ is model of $ZFC$ theory then (by replacing axiom) for every $f\in \cf$ $f[dom(f)]\in V[G_{B_\cf}].$ Let observe that for every $\xi \in \w_2\setminus B_\cf$ $c_\xi\notin V[G_{B_\cf}]$ (here $V[G] = V[G_{B_\cf}][G_{B_{\w_2\setminus B_\cf}}]$). Finally, for each $f\in\cf$ $A\setminus f[dom(f)] \ne \emptyset$ and then $A\setminus f[B] \ne \emptyset$ also. 
\end{proof}

Using fact that every Sierpiński set is strongly meager set (\cite{pawlikowski-archive}), we can prove in the analogous way theorem to the Theorem \ref{main-osipov}.

\begin{theorem}\label{main-osipov-stongly-meager} It is consistent with $ZFC$ theory that there exists $A\subseteq \bbr$ and $B\subseteq\bbr$ such that
\begin{itemize}
    \item $A$ is strongly meager,
    \item $B$ is completely $\cn$-nonmeasurable.
\end{itemize}
and for every family of continuous functions $\cf$, $|\cf|<\c$ satisfying $\forall f\in\cf$ $B\subseteq dom(f)$ we have that $A\setminus \bigcup\{ f[B]:\; f\in\cf\}\ne\emptyset.$ 
\end{theorem}


Now we move into subject of Eggleston Theorem, Fubini and Kuratowski-Ulam Theorem. Here we recall Eggleston Theorem only, see \cite{Brodskii} and \cite{Eggleston}.

\begin{theorem}[Eggleston] Let $m,n\in$ be positive integers and $\ci$ be one of them $\cm$ or $\cn$ $\sigma$-ideals defined on $\bbr^k$ euclidean space where $k\in\{ m,n,m+n \}$. Let us assume that $B\subseteq \bbr^{m+n}$ relation such that $B^c \in \ci$. Then there are perfect subsets $P\subseteq \bbr^m$ and $Q\subseteq \bbr^n$ such that $Q\notin \ci$ and $P\times Q \subseteq B.$
\end{theorem}

In \cite{Szymon} Żeberski proved a some generalisation of Egglestone Theorem for the real square $[0,1]^2.$

\begin{theorem}[Żeberski] Let $A\subseteq [0,1]^2$ be a comeager set. Then we can find two sets $G,Q\subseteq [0,1]$ such that $G$ is dense $G_\delta$ set, $Q$ is perfect set and $G\times Q\subseteq A.$
\end{theorem}

Here we consider little bit more general case. Namely, let $X$ be a compact Polish space and $\ch{(X)},$ be a family of all homeomorphisms of $X$ endowed with compact-open topology. Then $\ch{(X)}$ forms Polish space and then any $G_\delta$ uncountable subset of $\ch{(X)}$ is Polish space also. It is well known that the all cardinal coefficients for $\sigma$-ideal of the meager subsets does not depend of choice of Polish space.
\begin{theorem}\label{gener_meager} Let $X$ be a compact Polish space, $G\subseteq \ch(X)$ be uncountable $G_\delta$ set and $B\subseteq X$ be a comeager subset of $X$. Then 
\begin{enumerate}
    \item there are dense $G_\delta$ set $W\subseteq X$ and perfect set $Q\subseteq G$ such that  for every homeomorphism $f\in Q$ of $X$ we have $W\subseteq f[B]$,
    \item there are dense $G_\delta$ set $W\subseteq G$ and perfect set $P\subseteq X$ such that  for every homeomorphism $f\in W$ we have $P\subseteq f[B]$.
\end{enumerate}
\end{theorem}
\begin{proof} We prove only the first statement as the proof of the second one is very similar. Let us consider the Polish space $Z=G\times X \subseteq \ch(X)\times X$ with a product topology. Define the map $F: Z \mapsto Z$ as follows
	$$
	Z = G\times X \ni (f,x) \mapsto F(f,x) = (f,f(x)) \in Z.
	$$
	It is easy to see that $F$ is homeomorphism.
	Now, let us observe that by Kuratowski-Ulam Theorem the set $G\times B$ is comeager subset of $G\times X$ and then the set $F[G\times B]$ is also comeager in $G\times X$.
	The rest of this proof is follows from the proof of Eggleston Theorem on the real plane $\bbr^2$ given by Żeberski, see \cite{Szymon} (for the other proof, see \cite{MRZ}).
	Let us assume that we are working in transitive ground model of ZFC theory, say $V$. Now let us consider the generic extension $V'$ of $V$ in which $\cof(\cm) = \w_1 < \w_2=\c$ (for example let us iterate with countable support $\w_2$ times of Sacks forcing). Let us pass to $V'$ universe. By Cichoń-Kamburelis-Pawlikowski Theorem see \cite{CKP}, there is a family $\cf$ of the size $\w_1$ of dense $G_\delta$ subsets of $X,$ which is cofinal in $\{B\in Bor(X):X\setminus B\in \cm\}$. 
	By Kuratowski-Ulam Theorem the set $$
	\{ f\in G:\; ((F[G\times B])^c)_f\in\cm \} 
	= \{ f\in G:\; (f[B])^c\in\cm\}
	$$
	is comeager in $G$ and then has size $\c = \w_2.$ But our cofinal family $\cf$ has size $\w_1.$ Then there is a $G_\delta$ set $W\in \cf$ such that 
	$$
	A = \{ f\in G:\; W\subseteq f[B]\}
	$$
	has size $\w_2.$ Note that $A$ as a co-analityc set is a union of $\omega_{1}$ borel sets, thus there must be a perfect set $Q\subseteq A.$ Then we have $W\times Q \subseteq F[G\times B]$ what can be expressed as follows
	\begin{align*}
	&(\forall f\in G)(\forall y\in X) (f\in W\land y\in Q)\then (\exists x) (f,y) = (f,f(x))\land x\in B\\
	\iff&
	(\forall f\in G)(\forall x,y\in X) f\in W\land y\in Q\land (x,y)\in f \then x\in B
	\end{align*}
	So we have proved that
	$$
	(\exists W\in G_\delta)(\exists Q\in Perf(X)) W\times Q \subseteq F[G\times B]
	$$
	Let us observe that any dense $G_\delta$ set of $X$ can by coded by some real $c\in \w^\w$. This coding can be written by arithmetic formula. Finally,
	$$
	(\exists c\in\w^\w)(\exists Q\in Perf(X))(c\text{ coding dense }G_\delta)\land (\forall f\in G)(\forall x,y\in X)\; \varphi(c,x,y,f,B)
	$$
	is $\Sigma_2^1$-formula with ground model parameter $B$, where
	$$
	\varphi(c,x,y,f,B) \equiv f\in \#c\land y\in Q\land (x,y)\in f\then x\in B.
	$$
    The above formula is absolute between transitive models of $ZFC$ theory containing $\w_1$ by the Schoenfield's Theorem so it is satisfied in ground model $V$. 
\end{proof}

Here we can consider locally compact abelian Polish group $H$. This group $H$ can be equipped with Haar measure $\mu.$ Then the group $\ch_\mu(H)$ of all homeomorphism preserving measure $\mu$ forms a closed subspace of $\ch(H)$ of all homeomorphisms of $H.$ Then the cardinal coefficients for $\sigma$-ideal of $\mu$-null subsets are the same as for $\cn$ with respect to Lebesgue measure. Then we can almost rewrite the proof of the analogous theorem as Theorem \ref{gener_meager} in measure $\mu$ case.

\begin{theorem}\label{gener_haar} Let $H$ be a compact Polish abelian group equiped with Haar measure $\mu$ and $G$ be uncountable $G_\delta$ subset of $\ch_\mu(X)$. Let $B$ be a co-null subset of $H$. Then
\begin{enumerate}
    \item there are perfect subset $P\subseteq H$ and co-null $F_\sigma$ set $W\subseteq G$ such that  for every homeomorphism $f\in W$ of $H$ we have $P\subseteq f[B]$,
    \item there is a perfect set $Q\subseteq G$ and co-null $F_\sigma$ set $W\subseteq H$ such that  for every homeomorphism $f\in Q$ of $H$ we have $W\subseteq f[B]$.
\end{enumerate}
\end{theorem}

Here we can write some corrolaries of the Theorem \ref{gener_meager} and Theorem \ref{gener_haar}.

\begin{cor} For any Polish space $X$ let $B\subseteq X$ be comeager and $G\subseteq \ch(X)$ is the $G_\delta.$ Then there are two sets $W\subseteq X$ and $Q\subseteq G$ where $Q$ is a perfect set, $W$ is dense $G_\delta$ such that 
$$
W\subseteq \bigcap_{f\in Q} f[B].
$$
Moreover, we have $P\subseteq \bigcap_{f\in W} f[B],$ for some dense $G_\delta$ set $W\subseteq G$ and perfect set $P\subseteq X.$
\end{cor}

\begin{cor} Let $H$ be any Polish space and $B\subseteq H$ be co-null set and $G\subseteq \ch(X)$ is the $G_\delta.$ Then there are two sets $W\subseteq H$ and $Q\subseteq G$ where $Q$ is a perfect set, $W$ is dense $F_\sigma$ set such that 
$$
W\subseteq \bigcap_{f\in Q} f[B].
$$
Moreover, we have $P\subseteq \bigcap_{f\in W} f[B],$ for some co-null $F_\sigma$ set $W\subseteq G$ and perfect set $P\subseteq H.$
\end{cor}

The above theorems give us a simple corollary for euclidean $n$-dimensional space.
\begin{theorem} Let $n\in\w\setminus\{0,1\}$ and $B_n\subseteq \bbr^n$ be a $n$-dimensional unit ball in euclidean space with centred in origin coordinates. Let us assume that $E \subseteq B_{n}$ a comeager (or $B_n\setminus E$ is null) set in $B_n$. Then there are perfect set  $R$ in $n-1$-dimensional  boundary $D=bd(B_n)$, non-meager (non-null)  $P\subseteq B_{n-1}$ and $Q \subseteq [-1,1]$ such that
	$$
	(\forall \alpha \in R) \; (r_\alpha [P\times Q] \subseteq B_n),
	$$
	where $r_\alpha$ is rotation of $\alpha$ to the vector $(1,0,\ldots,0)\in\bbr^n$  over origin of the euclidean space $\bbr^n$.
\end{theorem}
In particular we have the following Proposition.
\begin{proposition} Let $D\subseteq \bbr^2$ be a unit disc with centred in origin coordinates and $B \subseteq D$ a comeager (or $D\setminus B$ is null) set in $D$. Then there are perfect set of directions $R$ on $bd(D)$ and $P,Q \subseteq [-1,1]$ such that
	$$
	(\forall \alpha \in R) \; (r_\alpha [P\times Q] \subseteq B),
	$$
	where $r_\alpha$ is rotation by $\alpha$ over origin of the real plane $\bbr^2$.
\end{proposition}


\begin{thebibliography}{123}
  \bibitem{Brodskii} M. L. Brodskii, On some properties of sets positive measure, Uspekhi Mat. Nauk, (1949), vol 4, pp. 136-138 (in Russian).
  \bibitem{five-poles} J. Cichoń, M. Morayne, R. Rałowski, C. Ryll-Nardzewski, S. Żeberski, On nonmeasurable unions, Topology and its Applications, 154 (2007), pp.884-893.
  \bibitem{CKP} J. Cichoń, A. Kamburelis, J. Pawlikowski, On dense subsets of the measure algebra, Proc. Amer. Math. Soc. 84 (1985), pp. 142-146.
\bibitem{CMR} J. Cichoń, M. Morayne, R. Rałowski, Images of Bernstein sets via continuous functions, Georgian Mathematical Journal 26(4) (2019), pp. 499-503. 
\bibitem{Eggleston} H. G. Eggleston, Two measure properties of Cartesian product sets, Quart. J. Math. Oxford Ser. (2) 5 (1954), pp. 108-115.
\bibitem{GMS} F. Galvin, J. Mycielski, R. M. Solovay, Strong measure zero sets, Notices
of the American Mathematical Society Vol. 26, No. 3 (1979), A-280
  \bibitem{mathoverflow-1} Mathoverflow: mathoverflow.net/questions/71976/lebesgue-non-measurability-in-the-plane, \url{https://mathoverflow.net/questions/71976/lebesgue-non-measurability-in-the-plane}
 \bibitem{MRZ} M. Michalski, R. Rałowski, Sz. Żeberski, Mycielski among trees, arXiv:1905.09069 {\url{https://arxiv.org/pdf/1905.09069.pdf}}.
 \bibitem{pawlikowski-archive} J. Pawlikowski, Every Sierpiński set is strongly meager, Arch. Math. Logic 35 (1996), 281–285.
 \bibitem{Szymon} Sz. Żeberski, Nonstandard proofs of Eggleston like Theorems, Proceedings of the Ninth Prague Topological Symposium (Prague 2001) pp.353-357, Topology Atlas, Toronto 2002.
 \bibitem{AO} private communication

 
 
 
\end{thebibliography}
\end{document}